\documentclass[11pt, amsfonts, dvipdfmx]{amsart}

\usepackage{amsmath,amssymb,amsthm,color}
\usepackage[dvipdfmx,colorlinks=true]{hyperref}
\usepackage[all]{xy}
\usepackage{tikz}

\numberwithin{equation}{section}

\SelectTips{eu}{12}

\newtheorem{theorem}{Theorem}[section]

\newtheorem{lemma}[theorem]{Lemma}
\newtheorem{corollary}[theorem]{Corollary}

\theoremstyle{definition}

\theoremstyle{remark}
\newtheorem{remark}[theorem]{Remark}

\newcommand{\Z}{\mathbb{Z}}

\newcommand{\R}{\mathbb{R}}

\renewcommand{\P}{\mathcal{P}}
\newcommand{\sam}[2]{\langle #1,#2 \rangle}

\title{Note on Samelson products in exceptional Lie groups}

\author{Daisuke Kishimoto}
\address{Department of Mathematics, Kyoto University, Kyoto, 606-8502, Japan}
\email{kishi@math.kyoto-u.ac.jp}

\author{Akihiro Ohsita}
\address{Faculty of Economics, Osaka University of Economics, Osaka 533-8533, Japan}
\email{ohsita@osaka-ue.ac.jp}

\author{Masahiro Takeda}
\address{Department of Mathematics, Kyoto University, Kyoto, 606-8502, Japan}
\email{m.takeda@math.kyoto-u.ac.jp}

\subjclass[2010]{55Q15, 57T10}
\keywords{exceptional Lie group, Samelson product, mod $p$ decomposition}

\begin{document}

\baselineskip.525cm

\maketitle

\begin{abstract}
  We determine (non-)triviality of Samelson products of inclusions of factors of the mod $p$ decomposition of $G_{(p)}$ for $(G,p)=(E_7,5),(E_7,7),(E_8,7)$. This completes the determination of (non-)triviality of those Samelson products in $p$-localized exceptional Lie groups when $G$ has $p$-torsion free homology.
\end{abstract}

%%%%% Section 1 %%%%%

\section{Introduction}

Let $X$ be a homotopy associative H-space with inverse. Recall that the Samelson product of maps $\alpha\colon A\to X$ and $\beta\colon B\to X$ is defined by
$$\sam{\alpha}{\beta}\colon A\wedge B\to X,\quad(a,b)\mapsto\alpha(a)\beta(b)\alpha(a)^{-1}\beta(b)^{-1}.$$
Samelson products are fundamental in the study of H-spaces, and have been studied intensely. %Suppose that an H-space $X$ decomposes into a product
%$$X\simeq X_1\times\cdots\times X_n.$$
%Then the most important Samelson products in $X$ are those of inclusions $X_i\to X$, which we call basic Samelson products. In this paper, we study basic Samelson products in $p$-localized exceptional Lie groups.
Let $G$ be a compact connected Lie group. It is well known that if we localize at the prime $p$, then $G$ admits a product decomposition
$$G\simeq_{(p)}B_1\times\cdots\times B_{p-1}$$
whenever $H_*(G;\Z)$ is $p$-torsion free (see \cite{MNT}). This decomposition is called the mod $p$ decomposition of $G$. Clearly, the most important Samelson products in a $p$-localized Lie group $G_{(p)}$ are those of inclusions $B_i\to G_{(p)}$, which we call \emph{basic Samelson products}.

Basic Samelson products in $p$-localized exceptional Lie groups are studied in \cite{HK2,HKMO,HKST,KK,KO}, and in particular, their (non-)triviality is completely determined for quasi-$p$-regular exceptional Lie groups, where each $B_i$ in the mod $p$ decomposition is an H-space of rank $\le 2$. Then the only remaining cases are as in the following table.

\renewcommand{\arraystretch}{1.4}

\begin{table}[htbp]
  \label{decomp}
  \centering
  %\caption{}
  \begin{tabular}{p{2.5cm}p{2cm}l}
    \hline
    Lie group&Prime&Mod $p$ decomposition\\\hline
    $E_7$&$p=5$&$B(3,11,19,27,35)\times B(15,23)$\\
    &$p=7$&$B(3,15,27)\times B(11,23,35)\times S^{19}$\\
    $E_8$&$p=7$&$B(3,15,27,39)\times B(23,35,47,59)$\\\hline
  \end{tabular}
\end{table}

Here $B(n_1,\ldots,n_r)$ is an indecomposable space such that
$$H^*(B(n_1,\ldots,n_r);\Z/p)=\Lambda(x_1,\ldots,x_r),\quad|x_i|=n_i.$$
In this paper, we determine (non-)triviality of basic Samelson products in the above three cases to complete the determination of (non-)triviality of basic Samelson products in $p$-localized exceptional Lie groups.

\begin{theorem}
  \label{main E_7}
  For $p=5,7$, all basic Samelson products in $(E_7)_{(p)}$ are non-trivial.
  %\begin{enumerate}
  %  \item Let $\epsilon_1\colon B(3,11,19,27,35)\to(E_7)_{(5)}$ and $\epsilon_2\colon B(15,23)\to(E_7)_{(5)}$ denote inclusions. Then for $1\le i,j\le 2$
  %  $$\sam{\epsilon_i}{\epsilon_j}\ne 0.$$
  %  \item Let $\epsilon_1\colon B(3,15,27)\to(E_7)_{(7)}$, $\epsilon_2\colon B(11,23,35)\to(E_7)_{(7)}$ and $\epsilon_3\colon S^{19}\to(E_7)_{(7)}$ denote inclusions. Then for $1\le i,j\le 3$
  %  $$\sam{\epsilon_i}{\epsilon_j}\ne 0.$$
  %\end{enumerate}
\end{theorem}

\begin{theorem}
  \label{main E_8}
  Let $\epsilon_1\colon B(3,15,27,39)\to(E_8)_{(7)}$ and $\epsilon_2\colon B(23,35,47,59)\to(E_8)_{(7)}$ denote inclusions. Then
  $$\sam{\epsilon_1}{\epsilon_1}=0,\quad\sam{\epsilon_1}{\epsilon_2}\ne 0,\quad\sam{\epsilon_2}{\epsilon_2}\ne 0.$$
\end{theorem}

\textit{Acknowledgement:} The first author is supported by JSPS KAKENHI No. 17K05248.

%%%%% Section 2 %%%%%

\section{Mod 7 cohomology of $BE_8$}

Throughout this paper, cohomology is taken over the field $\Z/p$, where $p$ is an odd prime.

Let $p$ be a prime $>5$. Then the mod $p$ cohomology of the classifying space $BE_8$ is given by
$$H^*(BE_8)=\Z/p[x_4,x_{16},x_{24},x_{28},x_{36},x_{40},x_{48},x_{60}],\quad|x_i|=i$$
and the mod $p$ cohomology of the classifying space $BSpin(2m)$ is given by
$$H^*(BSpin(2m))=\Z/p[p_1,\ldots,p_{m-1},e_m]$$
where $p_i$ and $e_m$ denote the Pontrjagin class and the Euler class, respectively. As in \cite{A}, there is a natural map $i\colon Spin(16)\to E_8$. In \cite{HKMO,HKO}, an explicit choice of generators $x_i$ of the mod $p$ cohomology of $BE_8$ are made through this natural map. In particular, one has:

\begin{lemma}
  \label{generator E_8}
  Generators $x_i$ of the mod $p$ cohomology of $BE_8$ can be chosen such that
  \begin{align*}
    i^*(x_4)&=p_1\\
    i^*(x_{16})&= 12p_4-\frac{18}{5}p_3p_1+p_2^2+\frac{1}{10}p_2p_1^2+168e_8\\
    i^*(x_{24})&\equiv 60p_6-5p_5p_1-5p_4p_2+3p_3^2+\frac{5}{36}p_2^3+110p_2e_8&&\mod(p_1^2)\\
    i^*(x_{28})&\equiv 480p_7+40p_5p_2-12p_4p_3-p_3p_2^2+312p_3e_8&&\mod(p_1)\\
    i^*(x_{36})&\equiv 480p_7p_2+72p_6p_3-30p_5p_4&&\mod(p_1)+I^3\\
    i^*(x_{40})&\equiv 480p_7p_3+50p_5^2&&\mod(p_1)+I^3\\
    i^*(x_{48})&\equiv -200p_7p_5-60p_7p_3p_2+3p_6p_3^2&&\mod (p_1)+I^4\\
    i^*(x_{60})&\equiv 144p_7p_5p_3-5p_5^3+\frac{3}{2}p_5^2p_3p_2-\frac{89}{1440}p_5p_4^2p_2-\frac{229}{1600}p_5p_4p_3^2&&\mod (p_1)+I^5
  \end{align*}
  where $I=\widetilde{H}^*(BSpin(16))$.
\end{lemma}

If a polynomial $P$ includes a monomial $M$, then we write $P>M$.

\begin{lemma}
  \label{P^1 E_8}
  In $H^*(BE_8)$ for $p=7$, $\P^1x_{40}$ is decomposable such that
  $$\P^1x_{40}>-3x_{28}x_{24}.$$
\end{lemma}

\begin{proof}
  By a degree reason, $\P^1x_{40}$ is decomposable. Recall from \cite{S} that there is the mod $p$ Wu formula in $H^*(BSpin(2m))$
  \begin{multline}
    \label{mod p Wu formula}
    \P^1p_n=\sum_{i_1+2i_2+\cdots+mi_m=n+\frac{p-1}{2}}(-1)^{i_1+\cdots+i_m+\frac{p+1}{2}}\frac{(i_1+\cdots+i_m-1)!}{i_1!\cdots i_m!}\\
    \times\left(2n-1-\frac{\sum_{j=1}^{n-1}(2n+p-1-2j)i_j}{i_1+\cdots+i_m-1}\right)p_1^{i_1}\cdots p_m^{i_m}
  \end{multline}
  where $p_m=e_m^2$. Then $\P^1x_{16}$ is decomposable. Then $\P^1(p_7p_3)>-5p_7p_6$, and so by Lemma \ref{generator E_8}, $\P^1i^*(x_{40})>-2400p_7p_6$. Thus by Lemma \ref{generator E_8}, $\P^1x_{40}>-3x_{28}x_{24}$.
\end{proof}

We make an alternative choice of generators of the mod 7 cohomology of $BE_8$.

\begin{lemma}
  \label{E_8 mod 7 cohomology}
  The mod 7 cohomology of $BE_8$ is given by
  $$H^*(BE_8)=\Z/7[y_4,y_{16},y_{24},y_{28},y_{36},y_{40},y_{48},y_{60}],\quad|y_i|=i$$
  such that
  $$\P^1y_{24}=y_{36},\quad\P^1y_{36}=y_{48},\quad\P^1y_{48}=y_{60},\quad\P^1y_{60}>2y_{48}y_{24}.$$
\end{lemma}

\begin{proof}
  By Lemma \ref{generator E_8}, we can choose $y_{24}$ such that $i^*(y_{24})$ does not include $p_4p_1^2$. Then we choose $y_{24}$ such that
  $$i^*(y_{24})\equiv 60p_6-5p_5p_1-5p_4p_2+3p_3^2+\frac{5}{36}p_2^3+110p_2e_8\mod I^4+(p_1^2)$$
  where $I=\widetilde{H}^*(BSpin(16))$. As in \cite{MNT}, the mod 7 cohomology of $B(23,35,47,59)$ is given by
  $$H^*(B(23,35,47,59))=\Lambda(z_{23},\P^1z_{23},(\P^1)^2z_{23},(\P^1)^3z_{23}),\quad|z_{23}|=23.$$
  Then we can choose $y_{36}=\P^1y_{24},\,y_{48}=\P^1y_{36},\,y_{60}=\P^1y_{48}$. In particular, we can see from \eqref{mod p Wu formula} that
  \begin{align*}
    i^*(y_{48})&\equiv p_7p_5+p_7p_4p_1+p_7p_3p_2+2p_6^2+2p_6p_5p_1\\
    &\quad+2p_6p_4p_2+4p_6p_3^2+2p_5^2p_2-2p_5p_4p_3-p_4^3\\
    \P^1i^*(y_{60})&\equiv -2p_7^2p_4+p_7p_6p_5+3p_6^3&\mod I^4+(e_8).
  \end{align*}
  Thus by looking at the term $p_7p_6p_5$ in $\P^1i^*(y_{60})$, we obtain $\P^1y_{60}>2y_{48}y_{24}$.
\end{proof}

%%%%% Section 3 %%%%%

\section{Mod 5 and 7 cohomology of $BE_7$}

In \cite{HK2,HKO,HKMO,KK}, informations of the mod $p$ cohomology of $BE_7$ for $p>5$ is obtained from $BE_8$ through the inclusion $E_7\to E_8$. However, we need the mod 5 cohomology of $BE_7$, and so we calculate the mod $p$ cohomology of $BE_7$ for $p>3$ from full scratch. Let $e_1,\ldots,e_8$ denote the dual of the standard basis of $\R^8$. Then the Dynkin diagram of $E_7$ is given by
\begin{center}
  \begin{tikzpicture}[x=0.7cm, y=0.7cm, thick]
    \draw (0,0) circle [radius=0.1];
    \node at (0,0.5) {$\alpha_1$};
    \draw (2,0) circle [radius=0.1];
    \node at (2,0.5) {$\alpha_3$};
    \draw (4,0) circle [radius=0.1];
    \node at (4,0.5) {$\alpha_4$};
    \draw (6,0) circle [radius=0.1];
    \node at (6,0.5) {$\alpha_5$};
    \draw (8,0) circle [radius=0.1];
    \node at (8,0.5) {$\alpha_6$};
    \draw (10,0) circle [radius=0.1];
    \node at (10,0.5) {$\alpha_7$};
    \draw (4,-2) circle [radius=0.1];
    \node at (3.4,-2) {$\alpha_2$};
    \draw (0.1,0)--(1.9,0);
    \draw (2.1,0)--(3.9,0);
    \draw (4.1,0)--(5.9,0);
    \draw (6.1,0)--(7.9,0);
    \draw (8.1,0)--(9.9,0);
    \draw (4,-0.1)--(4,-1.9);
  \end{tikzpicture}
\end{center}
such that
$$\alpha_1=\frac{1}{2}(e_1+e_8)-\frac{1}{2}(e_2+e_3+e_4+e_5+e_6+e_7),\quad\alpha_2=e_1+e_2,\quad\alpha_i=e_{i-1}-e_{i-2}\quad(3\le i\le 7)$$
where the Lie algebra of a maximal torus $T$ of $E_7$ is identified with $\{(x_1,\ldots,x_8)\in\R^8\mid x_7+x_8=0\}$.

Let $p$ be a prime $>3$. Then there is an isomorphism
$$H^*(BE_7)\cong H^*(BT)^W$$
where $W$ denotes the Weyl group of $E_7$. Hence we calculate the invariant ring of $W$ on the right hand side. Since the center of $E_7$ is isomorphic with $\Z/2$, $\alpha_1,\ldots,\alpha_7$ generates $H^*(BT)$. Set $t_1=-e_1$, $t_8=-e_8$ and $t_i=e_i$ for $2\le i\le 7$. Then one gets
$$H^*(BT)=\Z/p[t_1,\ldots,t_8]/(t_7-t_8).$$
Let $\varphi_i$ denote the reflection corresponding to $\alpha_i$ for $1\le i\le 7$. Then the Weyl group $W$ is generated by $\varphi_1,\ldots,\varphi_7$. Let $G$ be a subgroup of $W$ generated by $\varphi_i$ for $2\le i\le 7$. Then
$$H^*(BT)^G=\Z/p[t_7,t_8,p_1,\ldots,p_5,e_6]/(t_7-t_8)$$
where $p_i$ is the $i$-th symmetric polynomial in $t_1^2,\ldots,t_6^2$ and $e_6=t_1\cdots t_6$, and so
$$H^*(BT)^W=(\Z/p[t_7,t_8,p_1,\ldots,p_5,e_6]/(t_7-t_8))^{\varphi_1}.$$
Thus we investigate the invariant ring on the right hand side.

We define elements of $H^*(BT)^G$ by
\begin{align*}
  \bar{x}_4&=p_1\\
  \bar{x}_{12}&=-6p_3+p_2p_1+60e_6\\
  \bar{x}_{16}&=12p_4+p_2^2-\frac{1}{2}p_2p_1^2-36p_1e_6\\
  \bar{x}_{20}&=p_5-p_2e_6\\
  \bar{x}_{24}&=p_4p_2-\frac{1}{36}p_2^3-12p_3e_6+3p_2p_1e_6+48e_6^2\\
  \bar{x}_{28}&=p_5p_2-3p_4e_6-\frac{1}{4}p_2^2e_6\\
  \bar{x}_{36}&=p_5p_2^2-6p_4p_2e_6+36p_3e_6^2-\frac{1}{2}p_2^3e_6-72e_6^3
\end{align*}
which will turn out to represent generators of $H^*(BE_7)$. Let $c_i$ denote the $i$-the symmetric polynomial in $t_1,\ldots,t_8$ for $1\le i\le 8$, and let $q_i$ denote the $i$-the symmetric polynomial in $t_1^2,\ldots,t_8^2$ for $1\le i\le 8$. Then
$$q_i\equiv p_i,\quad c_6\equiv e_6\mod(t_7).$$
It is easy to see that
$$\varphi_1(t_i)=t_i-\frac{1}{4}(t_1+\cdots+t_8)$$
for $1\le i\le 8$. Then the ideal $(c_1,t_7)^2$ of $H^*(BT)$ is stable under the action of $\varphi_1$. Hence we calculate elements of $H^*(BT)^G$ which are invariant by $\varphi_1$ modulo $(c_1,t_7)^2$. Since
$$\sum_{i=0}^8\varphi_1(c_i)=\prod_{i=1}^8(1+\varphi_1(t_i))=\prod_{i=1}^8\left(1+t_i-\frac{1}{4}c_1\right)=\sum_{i=0}^8\left(1-\frac{1}{4}c_1\right)^{8-i}c_i$$
where $c_0=1$, one has
\begin{equation}
  \label{phi c}
  \begin{aligned}
    \varphi_1(c_1)&=-c_1,\qquad\varphi_1(c_2)=c_2,\\
    \varphi_1(c_i)&\equiv c_i-\frac{8-i+1}{4}c_{i-1}c_1\mod(c_1^2)\quad(3\le i\le 8).
  \end{aligned}
\end{equation}
Then, in particular, $\varphi_1(q_1)=q_1$, and so one gets:

\begin{lemma}
  \label{4}
  If $x\in H^4(BT)^G$ satisfies $\varphi_1(x)\equiv x\mod(c_1,t_7)^2$, then for some $a\in\Z/p$,
  $$x\equiv a\bar{x}_4\mod(t_7).$$
\end{lemma}

Since
\begin{equation}
  \label{q-c}
  q_i=c_i^2-2c_{i-1}c_{i+1}+2c_{i-2}c_{i+2}+\cdots+(-1)^i2c_{2i},
\end{equation}
it follows from \eqref{phi c} that
$$\varphi_1(q_i)\equiv q_i+d_ic_1\mod(c_1^2)$$
for $2\le i\le 7$ such that
\begin{equation}
  \label{phi}
  \begin{alignedat}{2}
    d_2&=\frac{3}{2}c_3\quad&d_3&=-\frac{1}{2}(5c_5+c_3c_2)\\
    d_4&=\frac{1}{2}(7c_7+3c_5c_2-c_4c_3)&d_5&=-\frac{1}{2}(5c_7c_2-3c_6c_3+c_5c_4)\\
    d_6&=-\frac{1}{2}(5c_8c_3-3c_7c_4+c_6c_5)\quad\qquad&d_7&=\frac{1}{2}(3c_8c_5-c_7c_6).
  \end{alignedat}
\end{equation}
One also has
$$\varphi_1(e_6)\equiv e_6-\frac{1}{4}c_5c_1\mod(c_1,t_7)^2.$$

\begin{lemma}
  \label{12,16}
  If $y_i\in H^i(BT)^G$ satisfies $\varphi_1(y_i)\equiv x\mod(c_1,t_7)^2$ for $i=12,16$, then for some $a,b,a',b',c'\in\Z/p$,
  \begin{align*}
    y_{12}&\equiv a\bar{x}_{12}+b\bar{x}_4^3\\
    y_{16}&\equiv a'\bar{x}_{16}+b'\bar{x}_{12}\bar{x}_4+c'\bar{x}_4^4\mod(t_7).
  \end{align*}
\end{lemma}

\begin{proof}
  Suppose that $y_{12}\in H^{12}(BT)^G$ satisfies $\varphi_1(y_{12})\equiv y_{12}\mod(c_1,t_7)^2$. Then we can set $y_{12}=a_1q_3+a_2q_2q_1+a_3e_6$ by Lemma \ref{4}. It follows from \eqref{q-c} and \eqref{phi} that
  $$\varphi_1(y_{12})\equiv y_{12}+\left(-\frac{5}{2}a_1-\frac{1}{4}a_3\right)c_5c_1+\left(-3a_2-\frac{1}{2}a_1\right)c_3c_2c_1\mod(c_1,t_7)^2.$$
  Then $y_{12}$ is a multiple of $\bar{x}_{12}$, and thus the first congruence holds.

  Suppose that $y_{16}\in H^{16}(BT)^G$ satisfies $\varphi_1(y_{16})\equiv y_{16}\mod(c_1,t_7)^2$. Then by Lemma \ref{4} and the first congruence, $y_{16}=a_1q_4+a_2q_2^2+a_3q_2q_1^2+a_4q_1e_6$. By \eqref{q-c} and \eqref{phi},
  $$\varphi_1(y)\equiv\left(\frac{3}{2}a_1+\frac{1}{2}a_4\right)c_5c_2c_1+\left(-\frac{1}{2}a_1+6a_2\right)c_4c_3c_1+\left(3a_2+6a_3\right)c_3c_2^2c_1\mod(c_1,t_7)^2.$$
  So $y_{16}$ is a multiple of $\bar{x}_{16}$. Thus one gets the second congruence.
\end{proof}

Since $\varphi(q_1)=q_1$, the ideal $(c_1,t_7)^2+(q_1^n)$ for $n\ge 1$ is stable under the action of $\varphi_1$. Then we can consider elements of $H^*(BT)$ which are invariant under the action of $\varphi_1$ modulo $(c_1,t_7)^2+(q_1^n)$.

\begin{lemma}
  \label{20,24}
  If $y_i\in H^i(BT)^G$ satisfies $\varphi_1(y_i)\equiv x\mod(c_1,t_7)^2+(q_1^2)$ for $i=20,24$, then for some $a,b,a',b',c'\in\Z/p$,
  \begin{align*}
    y_{20}&\equiv a\bar{x}_{20}+b\bar{x}_{16}\bar{x}_4\\
    y_{24}&\equiv a'\bar{x}_{24}+b'\bar{x}_{20}\bar{x}_4+c'\bar{x}_{12}^2\mod(t_7,p_1^2).
  \end{align*}
\end{lemma}

\begin{proof}
  Suppose that $y_{20}\in H^{20}(BT)^G$ satisfies $\varphi_1(y_{20})\equiv y_{20}\mod(c_1,t_7)^2$. Then by Lemmas \ref{4} and \ref{12,16}, we may set $y_{20}=a_1q_5+a_2q_3q_2+a_3q_2^2q_1+a_4q_2e_6$. Note that $c_2^2\in(c_1,t_7)^2+(q_1^2)$. Then by \eqref{q-c} and \eqref{phi},
  \begin{align*}
    \varphi_1(y)&\equiv\left(\frac{3}{2}a_1-3a_2+\frac{3}{2}a_4\right)c_6c_3c_1+\left(-\frac{1}{2}a_1-5a_2-\frac{1}{2}a_4\right)c_5c_4c_1\\
    &\quad+(-4a_2+12a_3)c_4c_3c_2c_1+\frac{3}{2}a_2c_3^3c_1\mod(c_1,t_7)^2+(q_1^2).
  \end{align*}
  Thus $y_{20}$ is a multiple of $\bar{x}_{20}$, and so the first congruence holds.

  Suppose that $y_{24}\in H^{24}(BT)^G$ satisfies $\varphi_1(y_{24})\equiv y_{24}\mod(c_1,t_7)^2$. Then by Lemmas \ref{4} and \ref{12,16} together with the first congruence, we can set $y_{24}=a_1q_4q_2+a_2q_3q_2q_1+a_3q_2^3+a_4q_3e_6+a_5q_2q_1e_6+a_6e_6^2$. By \eqref{q-c} and \eqref{phi},
  \begin{align*}
    \varphi_1(y)&\equiv y+\left(-2a_4-\frac{1}{2}a_6\right)c_6c_5c_1+\left(3a_1+6a_2-\frac{1}{2}a_4-3a_5\right)c_6c_3c_2c_1\\
    &\quad+\left(3a_1+10a_2+\frac{1}{2}a_4+a_5\right)c_5c_4c_2c_1+\left(-3a_1-\frac{1}{4}a_4\right)c_5c_3^2c_1\\
    &\quad+\left(\frac{1}{2}a_1+18a_3\right)c_4^2c_3c_1-3a_2c_3^3c_2c_1\mod(c_1,t_7)^2+(q_1^2).
  \end{align*}
  Then $y_{24}$ is a multiple of $\bar{x}_{24}$, and so the second congruence holds.
\end{proof}

\begin{remark}
  In \cite{HKO}, an element $-6p_3+p_2p_1-60e_6$ is chosen to be a $\varphi_1$-invariant by mistake, which forces to choose $p_5+p_2e_6$ as a $\varphi_1$-invariant. However, no problem occurs in \cite{HKO} because we do not use the term involving $e_6$.
\end{remark}

\begin{lemma}
  \label{28,36}
  If $y_i\in H^i(BT)^G$ satisfies $\varphi_1(y_i)\equiv y_i\mod(c_1,t_7)^2+(q_1)$ for $i=28,36$, then for some $a,b,a',b',c',d'\in\Z/p$,
  \begin{align*}
    y_{28}&\equiv a\bar{x}_{28}+b\bar{x}_{16}\bar{x}_{12}\\
    y_{36}&\equiv a'\bar{x}_{36}+b'\bar{x}_{24}\bar{x}_{12}+c'\bar{x}_{20}\bar{x}_{16}+d'\bar{x}_{12}^3\mod(t_7,p_1).
  \end{align*}
\end{lemma}

\begin{proof}
  Suppose that $y_{28}\in H^{28}(BT)^G$ satisfies $\varphi_1(y_{28})\equiv y_{28}\mod(c_1,t_7)^2+(q_1)$. Then by Lemmas \ref{4}, \ref{12,16} and \ref{20,24}, one may set $y_{28}=a_1q_5q_2+a_2q_3q_2^2+a_3q_4e_6+a_4q_2^2e_6$. Note that $c_2\in(c_1,t_7)^2+(q_1)$. Then by \eqref{q-c} and \eqref{phi},
  \begin{align*}
    \varphi_1(y_{28})&\equiv y_{28}+\left(-12a_2-\frac{1}{2}a_3+6a_4\right)c_6c_4c_3c_1+\left(\frac{3}{2}a_1+\frac{1}{2}a_3\right)c_5^2c_3c_1\\
    &\quad\left(-a_1-10a_2-\frac{1}{4}a_3-a_4\right)c_5c_4^2c_1+6a_2c_4c_3^3c_1\mod(c_1,t_7)^2+(q_1).
  \end{align*}
  Then $y_{28}$ is a multiple of $\bar{x}_{28}$, and thus the first congruence holds.

  Suppose that $y_{36}\in H^{36}(BT)^G$ satisfies $\varphi_1(y_{36})\equiv y_{36}\mod(c_1,t_7)^2+(q_1)$. Then by Lemmas \ref{4}, \ref{12,16} and \ref{20,24}, one may set $y_{36}=a_1q_5q_2^2+a_2q_3q_2^3+a_3q_4q_2e_6+a_4q_3^2e_6+a_5q_3e_6^2+a_6q_2^3e_6+a_7e_6^3$, and by \eqref{q-c} and \eqref{phi},
  \begin{align*}
    \varphi_1(y_{36})&\equiv y_{36}+\left(9a_4-\frac{3}{2}a_5-\frac{3}{4}a_7\right)c_6^2c_5c_1+\left(-3a_3-4a_4-\frac{1}{2}a_5\right)c_6c_5c_3^2c_1\\
    &\quad+\left(-6a_1-36a_2+\frac{1}{2}a_3+18a_6\right)c_6c_4^2c_3c_1+(6a_1+a_3)c_5^2c_4c_3c_1\\
    &\quad+\left(-2a_1-20a_2-\frac{1}{2}a_3-2a_6\right)c_5c_4^3c_1-\frac{1}{4}a_4c_5c_3^4c_1+18a_2c_4^2c_3^3c_1\\
    &\mod(c_1,t_7)^2+(q_1).
  \end{align*}
  So one can see that $y_{36}$ is a multiple of $\bar{x}_{36}$, completing the proof.
\end{proof}

As in \cite{A}, there is a map $j\colon Spin(12)\to E_7$ which is identified with the composite
$$H^*(BT)^W\xrightarrow{\rm incl}H^*(BT)^G\xrightarrow{\rm proj}\Z/p[p_1,\ldots,p_5,e_6].$$
Then by Lemmas \ref{4}, \ref{12,16}, \ref{20,24} and \ref{28,36}, one gets:

\begin{theorem}
  \label{BE_7}
  For $p>3$, the mod $p$ cohomology of $BE_7$ is given by
  $$H^*(BE_7)=\Z/p[x_4,x_{12},x_{16},x_{20},x_{24},x_{28},x_{36}]$$
  such that
  \begin{alignat*}{4}
    j^*(x_i)&=\bar{x}_i&&(i=4,12,16)\qquad&j^*(x_i)&\equiv\bar{x}_i\mod(p_1^2)\quad&(i=20,24)\\
    j^*(x_i)&\equiv\bar{x}_i\mod(p_1)\quad&&(i=28,36).
  \end{alignat*}
\end{theorem}

\begin{corollary}
  \label{P^1 E_7}
  \begin{enumerate}
    \item In $H^*(BE_7)$ for $p=5$, $\P^2x_{16},\,(\P^1)^2x_{24},\,(\P^1)^3x_{36}$ are decomposable such that
    $$(\P^1)^2x_{16}>-x_{20}x_{12},\quad\P^1x_{24}>2x_{16}^2,\quad(\P^1)^3x_{36}>x_{36}x_{24}.$$
    \item In $H^*(BE_7)$ for $p=7$, $\P^1x_{20}$, $\P^1x_{28}$ , $\P^1x_{36}$ are decomposable such that
    $$\P^1x_{20}>-2x_{28}x_4+x_{20}x_{12},\quad\P^1x_{28}>-2x_{28}x_{12}+4x_{20}^2,\quad\P^1x_{36}>x_{28}x_{20}.$$
  \end{enumerate}
\end{corollary}

\begin{proof}
  (1) By a degree reason, $(\P^1)^2x_{16},\;\P^1x_{24},\;(\P^1)^3x_{36}$ are decomposable. By \eqref{mod p Wu formula}, $\P^1p_4>3p_5p_1$ and $\P^1p_1>p_3$. One also has $\P^1e_6\in(e_6)$. Then by Theorem \ref{BE_7},
  $$\P^1j^*(x_{16})\equiv p_5p_1\mod(p_2,p_1^2,e_6)\quad\text{and}\P^1j^*(x_4)\equiv p_3\mod(p_2,p_1,e_6),$$
  so that $\P^1x_{16}>x_{20}x_4$ and $\P^1x_4>-x_{12}$. Thus $\P^2x_{16}>-x_{20}x_{12}$.

  By \eqref{mod p Wu formula}, $\P^1(p_4p_2)>3p_4^2$. Then by Theorem \ref{BE_7},
  $$\P^1j^*(x_{24})\equiv 3p_4^2\mod(p_2^2,p_1^2,e_6),$$
  implying $\P^1x_{24}>2x_{16}^2$.

  By \eqref{mod p Wu formula}, $\P^1(p_5p_2^2)>p_5p_4p_2+p_5p_2^3$. Then by Theorem \ref{BE_7} and $\P^1p_1\in(p_3,p_1)$,
  $$\P^1j^*(x_{36})\equiv p_5p_4p_2+p_5p_2^3\mod(p_3,p_1,e_6),$$
  implying $\P^1x_{36}>x_{28}x_{16}-x_{24}x_{20}$. By the same argument, we can show that $\P^1x_{16}>3x_{24},\;\P^1x_{20}>x_{28},\;\P^1x_{28}>3x_{36}$. Then $(\P^1)^3x_{36}>x_{36}x_{24}$.

  (2) By a degree reason, $\P^1x_{20}$ and $\P^1x_{28}$ are decomposable. By \eqref{mod p Wu formula}, $\P^1p_1\in(p_4,p_2^2,p_1)$, $\P^1p_5>-p_5p_3+p_5p_2p_1$ and $\P^1(p_5p_2-3p_4p_3)>4p_5^2+2p_5p_3p_2$. Since we can assume that $j^*(x_{20})$ does not include a multiple of $p_3p_1^2$, it follows from $\P^1e_6\in(e_6)$ and Theorem \ref{BE_7} that
  $$\P^1j^*(x_{20})\equiv-p_5p_3+p_5p_2p_1\mod(p_1^2,e_6),$$
  implying $\P^1x_{20}>2x_{28}x_4-x_{20}x_{12}$.

  By \eqref{mod p Wu formula}, $\P^1p_1\in(p_4,p_2^2,p_1)$ and $\P^1(p_5p_2-3p_4p_3)>4p_5^2+2p_5p_3p_2$. Since we may assume that $j^*(x_{28})$ does not include a multiple of $p_5p_3p_1$, it follows from $\P^1e_6\in(e_6)$ and Theorem \ref{BE_7} that
  $$\P^1j^*(x_{28})\equiv p_5^2p_2-p_5p_3p_2^2\mod(p_4,p_2^3,p_1,e_6)$$
  Then $\P^1x_{28}=2x_{28}x_{12}+4x_{20}^2$.

  By \eqref{mod p Wu formula}, $\P^1(p_5p_2^2)=p_5^2p_2$. Then by Theorem \ref{BE_7}, $\P^1x_{36}$ must include $x_{28}x_{20}$.
\end{proof}

As well as $E_8$, we make an alternative choice of generators of the mod 7 cohomology of $BE_7$.

\begin{lemma}
  \label{E_7 mod 7 cohomology}
  The mod 7 cohomology of $BE_7$ is given by
  $$H^*(BE_7)=\Z/7[y_4,y_{12},y_{16},y_{20},y_{24},y_{28},y_{36}],\quad|y_i|=i$$
  such that
  $$\P^1y_{12}=y_{24},\quad\P^1y_{24}=y_{36},\quad\P^1y_{36}>5y_{36}y_{12}.$$
\end{lemma}

\begin{proof}
  By Theorem \ref{BE_7}, we can choose $y_{12}$ such that $j^*(y_{12})=-6p_3+p_2p_1+60e_6$. As in \cite{MNT}, the mod 7 cohomology of $B(11,23,35)$ is given by
  $$H^*(B(11,23,35))=\Lambda(z_{11},\P^1z_{11},(\P^1)^2z_{11}),\quad|z_{11}|=11.$$
  Then we can choose $y_{24}$ and $y_{36}$ such that  $y_{24}=\P^1y_{12}$ and $y_{36}=\P^1y_{24}$. In particular, we can see from \eqref{mod p Wu formula} that
  \begin{align*}
    j^*(y_{36})&\equiv 3p_5p_2^2+2p_3^3&&\mod(p_1,e_6)\\
    \P^1j^*(y_{36})&\equiv 3p_5^2p_2+p_5p_3p_2^2+p_3^4+5p_3^2p_2^3&&\mod(p_1,p_4+3p_2^2,e_6).
  \end{align*}
  We may choose $y_{16}$ to be $x_{16}$ in Theorem \ref{BE_7}, so that $j^*(y_{16})=12p_4+p_2^2-\frac{1}{2}p_2p_1^2-36p_1e_6$. Thus by looking at the term $p_5p_3p_2^2$ in $\P^1j^*(y_{36})$, we obtain $\P^1y_{36}>5y_{36}y_{12}$.
\end{proof}

%%%%% Section 4 %%%%%

\section{Proofs of Theorems \ref{main E_7} and \ref{main E_8}}

We refine a criterion for non-triviality of Samelson products used in \cite{HK2,HKMO,HKO,HKST,KK}. See \cite{KT} for another refinement.

\begin{lemma}
  \label{criterion}
  Let $X$ be a simply-connected space such that
  $$H^*(X)=\Z/p[x_1,\ldots,x_n],\quad|x_1|<\cdots<|x_n|$$
  and let $\bar{\alpha}\colon\Sigma A\to X$ and $\bar{\beta}\colon \Sigma B\to X$ be maps. Given a cohomology operation $\theta$, suppose that there is $1\le i,j,k\le n$ satisfying the following conditions:
  \begin{enumerate}
    \item $\bar{\alpha}^*(x_i)\ne 0$ and $\bar{\beta}^*(x_j)\ne 0$;

    \item $\dim A=|x_i|-1$ and $\dim B=|x_j|-1$.

    \item $\theta x_k$ is decomposable such that $\theta x_k>cx_ix_j$ for $c\ne 0$;

    \item for any extension $\mu\colon\Sigma A\times\Sigma B\to X$ of $\bar{\alpha}\vee\bar{\beta}\colon\Sigma A\vee\Sigma B\to X$, $\theta\mu^*(x_k)=0$.
  \end{enumerate}
  Then
  $$\sam{\alpha}{\beta}\ne 0$$
  where $\alpha\colon A\to\Omega X$ and $\beta\colon B\to\Omega X$ are the adjoint of $\bar{\alpha}$ and $\bar{\beta}$, respectively.
\end{lemma}

\begin{proof}
  Suppose $\sam{\alpha}{\beta}=0$. Then since the Whitehead product $[\bar{\alpha},\bar{\beta}]$ is the adjoint of $\sam{\alpha}{\beta}$, $[\bar{\alpha},\bar{\beta}]=0$, implying that there is a homotopy commutative diagram
  $$\xymatrix{\Sigma A\vee\Sigma B\ar[r]^(.65){\bar{\alpha}\vee\bar{\beta}}\ar[d]&X\ar@{=}[d]\\
  \Sigma A\times\Sigma B\ar[r]^(.65)\mu&X}$$
  for some map $\mu$. So
  $$\mu^*(x_i)\equiv\bar{\alpha}^*(x_i),\quad\mu^*(x_j)\equiv\bar{\beta}^*(x_j)\mod\widetilde{H}^*(\Sigma A\times\Sigma B)^2.$$
  By $|x_1|<\cdots<|x_n|$ and the conditions (1) and (2), $\mu^*(y)=0$ for any decomposable monomial $y$ which is not a multiple of $x_ix_j$. Then it follows from the condition (3) that
  $$\mu^*(\theta x_k)=\mu^*(cx_ix_j)=c\mu^*(x_i)\mu^*(x_j)=c(\bar{\alpha}^*(x_i)\times\bar{\beta}^*(x_j))\ne 0.$$
  On the other hands, by the condition (4), $\theta\mu^*(x_k)=0$. Thus we obtain a contradiction, and therefore $\sam{\alpha}{\beta}\ne 0$.
\end{proof}

\begin{lemma}
  \label{E_7 p=5}
  Let $\epsilon_1\colon B(3,11,19,27,35)\to(E_7)_{(5)}$ and $\epsilon_2\colon B(15,23)\to(E_7)_{(5)}$ be the inclusions. Then for each $i,j$, $\langle\epsilon_i,\epsilon_j\rangle\ne 0$.
\end{lemma}

\begin{proof}
  Set $B_1=B(3,11,19,27,35)$ and $B_2=B(15,23)$, and let $\bar{\epsilon}_i\colon B_i\to(BE_7)_{(5)}$ be the adjoint of $\epsilon_i$ for $i=1,2$. Clearly,
  $$\bar{\epsilon}_1(x_i)\ne 0\quad(i=4,12,20,28,36)\quad\text{and}\quad\bar{\epsilon}_2(x_i)\ne 0\quad(i=16,24).$$
  By a degree reason, $(\P^1)^3H^{36}(\Sigma B_1^{(35)}\times\Sigma B_2^{(23)})=0$ and $\P^1H^{24}(S^{16}\times S^{16})=0$. Then by Corollary \ref{P^1 E_7}, we can apply Lemma \ref{criterion}, so that $\sam{\epsilon_1\vert_{B_1^{(35)}}}{\epsilon_2\vert_{B_2^{(23)}}}\ne 0$ and $\sam{\epsilon_2\vert_{S^{15}}}{\epsilon_2\vert_{S^{15}}}\ne 0$. Thus $\sam{\epsilon_1}{\epsilon_2}\ne 0$ and $\sam{\epsilon_2}{\epsilon_2}\ne 0$. Since $\P^1H^{11}(B_1)=0$ as in \cite{MNT}, we also have $(\P^1)^2H^{16}(\Sigma B_1^{(11)}\times \Sigma B_1^{(19)})=0$. Then by Corollary \ref{P^1 E_7}, we can also apply Lemma \ref{criterion}, so that $\sam{\epsilon_1\vert_{B_1^{(19)}}}{\epsilon_1\vert_{B_1^{(11)}}}\ne 0$. Thus $\sam{\epsilon_1}{\epsilon_1}\ne 0$, completing the proof.
\end{proof}

\begin{lemma}
  \label{E_7 p=7}
  Let $\epsilon_1\colon B(3,15,27)\to(E_7)_{(7)},\;\epsilon_2\colon B(11,23,35)\to(E_7)_{(7)},\;\epsilon_3\colon S^{19}\to(E_7)_{(7)}$ be the inclusions. Then for each $i,j$, $\langle\epsilon_i,\epsilon_j\rangle\ne 0$.
\end{lemma}

\begin{proof}
  Set $B_1=B(3,15,27)$ and $B_2=B(11,23,35)$, and let $\bar{\epsilon}_i$ denote the adjoint of $\epsilon_i$ for $i=1,2,3$. Clearly,
  $$\bar{\epsilon}_1(x_i)\ne 0\quad(i=4,16,28),\quad\bar{\epsilon}_2(x_i)\ne 0\quad(i=12,24,36),\quad\bar{\epsilon}_3^*(x_{20})\ne 0.$$
  Since $\P^1H^{28}(S^{20}\times S^{20})=0$, it follows from Corollary \ref{P^1 E_7} that we can apply Lemma \ref{criterion}, so that $\sam{\epsilon_3}{\epsilon_3}\ne 0$. Since $\P^1H^{20}(S^{12}\times S^{20})=0$, by Corollary \ref{P^1 E_7} and Lemma \ref{criterion}, we also get $\sam{\epsilon_2\vert_{S^{19}}}{\epsilon_3}\ne 0$, implying $\sam{\epsilon_2}{\epsilon_3}\ne 0$.

  As in \cite{MNT}, $\P^1H^{15}(B_1)=0$, and so
  $$\P^1H^{20+k}(\Sigma B_1^{(27)}\times S^{4+k})=0\quad(k=0,8,16).$$
  Then by Corollary \ref{P^1 E_7} and Lemma \ref{criterion}, $\sam{\epsilon_1\vert_{B_1^{(27)}}}{\epsilon_1\vert_{S^3}}\ne 0$, $\sam{\epsilon_1\vert_{B_1^{(27)}}}{\epsilon_2\vert_{S^{11}}}\ne 0$ and $\sam{\epsilon_1\vert_{B_1^{(27)}}}{\epsilon_3}$. Thus $\sam{\epsilon_1}{\epsilon_1}\ne 0$, $\sam{\epsilon_1}{\epsilon_2}\ne 0$ and $\sam{\epsilon_1}{\epsilon_3}\ne 0$.

  Consider the cohomology generators in Lemma \ref{E_7 mod 7 cohomology}. Note that $\bar{\epsilon}_1(y_{28})\ne 0$ and $\bar{\epsilon}_3(y_{20})\ne 0$. Let $\mu\colon\Sigma B_2^{(12)}\times\Sigma B_2^{(35)}\to(BE_7)_{(7)}$ be any extension of $\bar{\epsilon}_2\vert_{\Sigma B_2^{(23)}}\vee\bar{\epsilon}_2\vert_{\Sigma B_2^{(23)}}\colon\Sigma B_2^{(23)}\vee\Sigma B_2^{(23)}\to(BE_7)_{(7)}$. By a degree reason, $\mu^*(y_{12})=\Sigma u_{11}\times 1+1\times\Sigma u_{11}$, where $u_{11}$ is a generator of $H^*(B_2)$ of dimension 11. Then by Lemma \ref{E_7 mod 7 cohomology},
  $$\mu^*(y_{36})=(\P^1)^2\Sigma u_{11}\times 1+1\times(\P^1)^2\Sigma u_{11},$$
  implying $\P^1\mu^*(y_{36})=0$. On the other hand, $\P^1y_{36}$ is, clearly, decomposable. Then it follows from Lemmas \ref{E_7 mod 7 cohomology} and \ref{criterion} that $\sam{\epsilon_2\vert_{B_2^{(23)}}}{\epsilon_2\vert_{B_2^{(23)}}}\ne 0$. Thus $\sam{\epsilon_2}{\epsilon_2}\ne 0$.
\end{proof}

Now we are ready to prove Theorem \ref{main E_7}.

\begin{proof}[Proof of Theorem \ref{main E_7}]
  Combine Lemmas \ref{E_7 p=5} and \ref{E_7 p=7}.
\end{proof}

The homotopy groups of the factor spaces of the mod 7 decomposition of $E_8$ are calculated in \cite{MNT}. In particular, one has:

\begin{lemma}
  \label{homotopy group}
  $\pi_k(E_8)_{(7)}=0$ for $k=12i+6$ with $0\le i\le 6$.
\end{lemma}

From \cite[Theorem 3.1, Lemma 4.2]{T} one can easily deduce the following.

\begin{lemma}
  \label{retract}
  Let $B=B(3,15,27,39)$ be a factor space of the mod 7 decomposition of $E_8$. Then there is a subcomplex $A=S^3\cup e^{15}\cup e^{27}\cup e^{39}$ of $B$ such that the inclusion $\Sigma A\to\Sigma B$ has a left homotopy inverse $r\colon\Sigma B\to\Sigma A$ satisfying a homotopy commutative diagram
  $$\xymatrix{\Sigma B\ar[r]^j\ar[d]_r&(BE_8)_{(7)}\ar@{=}[d]\\
  \Sigma A\ar[r]^(.45){j\vert_{\Sigma A}}&(BE_8)_{(7)}}$$
  where the map $j$ is the adjoint of the inclusion $B\to(E_8)_{(7)}$.
\end{lemma}

We finally prove Theorem \ref{main E_8}.

\begin{proof}
  [Proof of Theorem \ref{main E_8}]
  Set $B_1=B(3,15,27,39)$ and $B_2=B(23,35,47,59)$. Let $A$ be a subcomplex of $B_1$ in Lemma \ref{retract}. Since $A\wedge A$ has cells in dimension $12i+6$ for $0\le i\le 6$, it follows from Lemma \ref{homotopy group} that $\sam{\epsilon_1\vert_A}{\epsilon_1\vert_A}=0$. Then the corresponding Whitehead product $[j\vert_{\Sigma A},j\vert_{\Sigma A}]$ is trivial so that the map $j\vert_{\Sigma A}\vee j\vert_{\Sigma A}\colon\Sigma A\vee\Sigma A\to(BE_8)_{(7)}$ extends over $\Sigma A\times\Sigma A$, where $j\colon\Sigma B_1\to (BE_8)_{(7)}$ denotes the adjoint of $\epsilon_1$ as in Lemma \ref{retract}. Thus there is a homotopy commutative diagram
  $$\xymatrix{\Sigma B_1\vee\Sigma B_1\ar[r]^{r\vee r}\ar[d]&\Sigma A\vee\Sigma A\ar[rr]^(.55){j\vert_{\Sigma A}\vee j\vert_{\Sigma A}}\ar[d]&&(BE_8)_{(7)}\ar@{=}[d]\\
  \Sigma B_1\times\Sigma B_1\ar[r]^{r\times r}&\Sigma A\times\Sigma A\ar[rr]&&(BE_8)_{(7)}}.$$
  By Lemma \ref{retract}, the composite of the top maps is $j\vee j$, implying $[j,j]=0$. Thus $\sam{\epsilon_1}{\epsilon_1}=0$.

  Since $\P^1H^{15}(B_1)=0$ as in \cite{MNT}, $\P^1H^{40}(\Sigma B_1^{(27)}\times\Sigma S^{24})=0$. Then by Lemmas \ref{P^1 E_8} and \ref{criterion}, $\sam{\epsilon_1\vert_{B_1^{(27)}}}{\epsilon_2\vert_{S^{23}}}\ne 0$, implying $\sam{\epsilon_1}{\epsilon_2}\ne 0$.

  Consider the cohomology generators in Lemma \ref{E_8 mod 7 cohomology}. Then $\bar{\epsilon}_2(y_i)\ne 0$ for $i=24,48$. Let $\mu\colon\Sigma B_2^{(47)}\times\Sigma B_2^{(23)}\to(BE_8)_{(7)}$ be any extension of $\bar{\epsilon}_2\vert_{\Sigma B_2^{(47)}}\vee\bar{\epsilon}_2\vert_{\Sigma B_2^{(23)}}\colon\Sigma B_2^{(47)}\vee\Sigma B_2^{(23)}\to(BE_8)_{(7)}$. By a degree reason, $\mu^*(y_{24})=\Sigma u_{23}\times 1+1\times\Sigma u_{23}$, where $u_{23}$ is a generator of $H^*(B_2)$ of dimension 23. Then by Lemma \ref{E_7 mod 7 cohomology},
  $$\mu^*(y_{60})=(\P^1)^3\Sigma u_{23}\times 1+1\times(\P^1)^3\Sigma u_{23},$$
  implying $\P^1\mu^*(y_{60})=0$. On the other hand, $\P^1y_{60}$ is decomposable by a degree reason. Then it follows from Lemmas \ref{E_8 mod 7 cohomology} and \ref{criterion} that $\sam{\epsilon_2\vert_{B_2^{(47)}}}{\epsilon_2\vert_{B_2^{(23)}}}\ne 0$. Thus $\sam{\epsilon_2}{\epsilon_2}\ne 0$.
\end{proof}

\end{document}